\documentclass[12pt,reqno]{amsart}

\usepackage[all,2cell]{xy} \UseAllTwocells \SilentMatrices
\usepackage{amssymb}

\newtheorem{theorem}{Theorem}[section]
\newtheorem{proposition}[theorem]{Proposition}
\newtheorem{lemma}[theorem]{Lemma}

\theoremstyle{definition}
\newtheorem{remark}[theorem]{Remark}

\numberwithin{equation}{section}

\DeclareMathOperator*{\Supp}{Supp}
\DeclareMathOperator*{\Vect}{Vect}
\DeclareMathOperator*{\PVect}{PVect}
\DeclareMathOperator*{\Gal}{Gal}
\DeclareMathOperator*{\lcm}{l.c.m.}

\newcommand{\cO}{\mathcal{O}}
\newcommand{\PP}{\mathbb{P}}

\begin{document}

\title[Genuinely ramified maps and parabolic bundles]{Genuinely ramified maps and
stability of pulled-back parabolic bundles}

\author[I. Biswas]{Indranil Biswas}

\address{School of Mathematics, Tata Institute of Fundamental
Research, Homi Bhabha Road, Bombay 400005, India}

\email{indranil@math.tifr.res.in}

\author[M. Kumar]{Manish Kumar}

\address{Statistics and Mathematics Unit, Indian Statistical Institute,
Bangalore 560059, India}

\email{manish@isibang.ac.in}

\author[A.J. Parameswaran]{A. J. Parameswaran}

\address{School of Mathematics, Tata Institute of Fundamental
Research, Homi Bhabha Road, Bombay 400005, India}

\email{param@math.tifr.res.in}



\begin{abstract}
Let $f\, :\, X\, \longrightarrow\, Y$ be a genuinely ramified map between irreducible
smooth projective curves defined over an algebraically closed field. Let $P$ be a branch
data on $Y$ such that $P(y)$ and $B_f(y)$, where $B_f$ is branch data for $f$, are linearly
disjoint for every $y\, \in\, Y$. Further assume that either $P$ is tame or $f$ is Galois.
Then the pullback, by $f$, of any
stable parabolic bundle on $Y$ with respect to $P$ is actually a stable parabolic
bundle on $X$ with respect to $f^*P$.
\end{abstract}
\maketitle

\section{Introduction}

Let $f\, :\, X\, \longrightarrow\, Y$ be a separable surjective map between irreducible
smooth projective curves defined over an algebraically closed field. In \cite{BP} the
following was proved:

If the homomorphism of \'etale fundamental groups $f_*\, :\,
\pi_1^{\rm et}(X)\,\longrightarrow\, \pi_1^{\rm et}(Y)$
induced by is surjective, and $E$ is a stable
vector bundle on $Y$, then the pullback $f^*E$ is also stable.

Any map $f$ as above such that the corresponding $f_*\, :\,
\pi_1^{\rm et}(X)\,\longrightarrow\, \pi_1^{\rm et}(Y)$ is surjective is called
a genuinely ramified map; see \cite[Proposition 2.6]{BP} and \cite[Lemma 3.1]{BP},
\cite[p.~574, Proposition 3.3]{BHS} for equivalent reformulation of it.

Given a smooth projective curve $Y$ defined over an algebraically closed field of 
characteristic zero, and some marked points $S$ of $Y$, parabolic bundles on $Y$ with 
parabolic structure over $S$ were introduced in \cite{MS} by Mehta and Seshadri. They 
carried out a detailed investigation of the parabolic bundles, and proved that the stable 
parabolic bundles of parabolic degree zero over a complex curve are in bijection with the 
irreducible unitary representations of the complement of the parabolic points of the 
curve. In \cite{Bi} a correspondence between the parabolic bundles and the orbifold 
bundles was established (see also \cite{Bo1} and \cite{Bo2}).

In \cite{KP} and \cite{KM} the definition of parabolic bundles was extended to smooth 
projective curves defined over an algebraically closed field of arbitrary characteristic. 
This was carried out using the formal orbifolds. A formal orbifold is a curve $Y$ together 
with a branch data $P$ on a finite subset $S\, \subset\, Y$, which comprises of a 
nontrivial Galois extension $P(y)$ of $K_{Y,y}$, the fraction field of $\widehat 
\cO_{Y,y}$ for each marked point $y_i\, \in\, S$. A parabolic structure on a vector bundle 
$V$ over $Y$ with branch data $P$, roughly speaking, consists of local equivariant structure 
on pullback of $V$ over the formal neighbourhood of $y_i$ which become trivial on the 
``punctured formal disc'', for all $y_i$ in $S$ (see Section \ref{sec2} for the precise 
definition).

A natural question to ask is whether an analogue of the above mentioned theorem of
\cite{BP} continues to hold in the more general set-up of parabolic bundles.

Let $f\, :\, X\, \longrightarrow\, Y$ be a genuinely ramified map between irreducible
smooth projective curves. Then $f$ produces a branch data $B_f$ on $Y$.

We prove the following theorem (see Theorem \ref{thm1}):

\begin{theorem}\label{thm-i}
Suppose $B_f(y)$ and $P(y)$ are linearly disjoint over $K_{Y,y}$ for all $y\,\in \,Y$.
Also assume that either $P$ is tame or the map $f$ is Galois. Then the pullback, by $f$,
of every stable parabolic bundle on $Y$ with respect to $P$ is a stable parabolic
bundle on $X$ with respect to $f^*P$.
\end{theorem}

\section{Parabolic bundles}\label{sec2}

The base field $k$ is algebraically closed. Let $Y$ be an irreducible smooth projective 
curve. We recall some definitions from \cite[Section 2]{KP}. For each closed point $y\,\in 
\,Y$, let $K_{Y,y}$ be the fraction field of the completion $\widehat \cO_{Y,y}$. A 
\textit{branch data} $P$ on $Y$ assigns a finite Galois extension $P(y)$ of $K_{Y,y}$ for 
every closed point $y\,\in\, Y$ such that $P(y)$ is a trivial extension of $K_{Y,y}$ for 
all $y$ outside some finitely many closed points $y_1,\,\cdots,\, y_n$. The collection 
$\{y_1,\,\cdots,\, y_n\}$ is called the support of $P$. If $n\,=\, 0$, i.e., the support 
is empty, then the branch data will be called trivial. The branch data $P$ is called 
\textit{tame} if each of these extension of local fields is a tame extension.

Let $f\,:\,X\,\longrightarrow\, Y$ be a separable (same as generically smooth) nonconstant 
morphism of smooth projective curves. We will call such a morphism a (ramified) 
\emph{covering}. It will be called a \emph{Galois covering} if the induced field extension 
$k(X)/k(Y)$ is Galois. Using $f$, we will construct a branch data $B_f$ on $Y$. For 
$y\,\in\, Y$, define $B_f(y)$ to be the compositum of the Galois closure of $K_{X,x}/K_{Y,y}$ for 
all $x\,\in\, X$ with $f(x)\,=\,y$. Therefore, the support of $B_f$ is the branch locus of $f$. We 
note that if $f$ is tamely ramified, then $B_f(y)$ is the cyclic field extension 
of $K_{Y,y}$ whose degree is the least common multiple of the multiplicities of $f$ at the 
points of $f^{-1}(y)$.

A \emph{formal orbifold curve} is a pair $(Y,\,P)$ where $Y$ is a smooth projective curve 
and $P$ is a branch data on $Y$. A morphism $f\,:\,(X,\,Q)\,\longrightarrow\, (Y,\,P)$ of formal 
orbifold curves is a finite separable morphism $f:X\,\longrightarrow\, Y$ such that for 
all $x\,\in\, X$ the field $Q(x)$ contains $P(f(x))$. This morphism is called \'etale if 
$Q(x)\,=\,P(f(x))$ for all $x\,\in\, X$. We say that $(Y,\,P)$ is \emph{geometric}, or $P$ is 
a \emph{geometric branch data}, if there exists an \'etale covering $f\,:\,(X,\,0)\,\longrightarrow\, 
(Y,\,P)$ with $0$ being the trivial branch data. Note that in this case $P=B_f$.

Fix a finite subset
$$
S\, :=\, \{y_1,\,\cdots ,\,y_n\}\, \subset\, Y
$$
of a smooth projective curve $Y$. Fix a branch data $P$ with support $S$.
For any $y\, \in\, S$, let 
$I_{y_i}\,=\, {\rm Gal}(P(y_i)\big\vert K_{Y,y_i})$ be the Galois group.
The integral closure of
$\cO_{Y,y_i}$ in $P(y_i)$ will be denoted by $R_i$.

Let $V$ be a vector bundle over $Y$; the stalk of $V$ (viewed as a locally free sheaf) at $y_i$ is denoted by $V_{y_i}$. 
We recall from \cite{KM} that a parabolic structure on $V$ over $S$ with
branch data $P$ consists 
of an $I_{y_i}$--equivariant structure on $V_{y_i}\otimes_{\cO_{Y,y_i}} R_i$ together with 
an isomorphism between the induced $I_{y_i}$--equivariant structure on 
$V_{y_i}\otimes_{\cO_{Y,y_i}} P(y_i)$ and the trivial $I_{y_i}$--equivariant structure on 
$V_{y_i}\otimes_{\cO_{Y,y_i}} P(y_i)$, for every $y_i\, \in\, S$. See \cite[Section 4]{KM} 
for more details.

Let $f\, :\, (X,\,0)\,\longrightarrow\, (Y,\,P)$ be an \'etale Galois covering of formal 
orbifold curves; the Galois group of $f$ will be denoted by $\Gamma$. We know that the 
category $\PVect(Y,P)$ of parabolic bundles on $Y$ with the branch data $P$ is equivalent 
to the category of $\Gamma$--equivariant vector bundles on $X$ \cite[Theorem 4.11]{KM}. In 
fact in \cite{KP} the category of vector bundles $\Vect(Y,P)$ on $(Y,\,P)$ was defined as 
$\Gamma$--equivariant bundles on $X$, and it was shown to be independent of the choice of 
the \'etale Galois covering $f$ \cite[Proposition 3.6]{KP}. Note that if $P'\,\ge \,P$, 
then there is an embedding of categories $$\PVect(Y,P)\, \longrightarrow \,\PVect(Y,P')$$ 
\cite[Corollary 5.10]{KM}. Furthermore, if $P$ and $P'$ are geometric, then there is an 
embedding $\Vect(Y,P) \, \longrightarrow \,\Vect(Y,P')$ \cite[Theorem 3.7]{KP}. An 
orbifold bundle on $Y$ is an object of $\Vect(Y,P)$ for some geometric branch data $P$, 
and a parabolic bundle on $Y$ is an object of $\PVect(Y,P)$. The functors between 
$\Vect(Y,P)$ and $\PVect(Y,P)$ induce an equivalence between the category of orbifold 
bundles and the category of parabolic bundles $\PVect(Y)$ on $Y$.

Let $\PVect^t(Y)$ be the full subcategory of $\PVect(Y)$ whose objects come from the 
objects in $\PVect(Y,P)$ for some tame branch data $P$ on $Y$. Similarly, we can define 
the full subcategory of tame orbifold bundles whose objects come from objects in 
$\Vect(X,P)$ for some tame branch data $P$. The restriction of the above functor to 
these full subcategories is an equivalence as well.

When the base field has characteristic zero, parabolic bundles were defined by Mehta and Seshadri (\cite{MS} as follows.
Let $E$ be a vector bundle on $Y$ of rank $r$.
Let $E_*$ be a parabolic structure on $E$ with parabolic divisor 
$$
S\, :=\, \{y_1,\,\cdots ,\,y_n\}\, \subset\, Y\, .
$$
For every $y_i\, \in\, S$, let
$$
E_{p_i}\,=\,E_{i,1}\,\supset\, E_{i,2}\,\supset\, \cdots\,
\supset \,E_{i,l_i} \,\supset\, E_{i,l_i+1}\,=\, 0
$$
be the quasiparabolic filtration, and let
$$
0\,\leq\, \alpha_{i,1} \,< \,\alpha_{i,2} \,<\,
\cdots \,<\, \alpha_{i,l_i}\,<\, 1
$$
be the corresponding parabolic weights (see \cite{MS}, \cite{Bi}). Then
the parabolic degree of $E_*$ is defined to be
$$
\text{par-deg}(E_*)\,:=\, \text{degree}(E)+\sum_{i=1}^n\sum_{j=1}^{l_i}
\dim (E_{i,j}/E_{i,j+1})\cdot \alpha_{i,j},
$$
and the parabolic slope of $E_*$ is defined to be
$$
\text{par-}\mu(E_*)\,:=\, \frac{\text{par-deg}(E_*)}{r}\, .
$$

The parabolic bundle $E_*$ is called \text{stable} (respectively, \text{semistable}) if
$$
\text{par-}\mu(F_*)\, <\, \text{par-}\mu(E_*)\, \ \ {\rm (respectively, }\,\,
\text{par-}\mu(F_*)\, \leq\, \text{par-}\mu(E_*){\rm )}
$$
for every subbundle $0\, \not=\, F\, \subsetneq \, E$ with the parabolic structure
induced by that of $E_*$ (see \cite{MS}, \cite{Bi}). The parabolic bundle $E_*$ is
called \text{polystable} if
\begin{itemize}
\item $E_*$ is semistable, and

\item $E_*$ is a direct sum of stable parabolic bundles.
\end{itemize}

When the base field has positive characteristic, the parabolic degree, slope, stability and semistability were defined in terms of equivariant bundles (see Remark 3.9 of \cite{KP} and the paragraph following the remark). In characteristic zero, the definition of (semi)stability in terms of equivariant bundles and the above definition agree (see Proposition \ref{prop3}).

In \cite[Proposition 5.15]{KM} it was shown that if the base field $k$ has characteristic 
zero then the category $\PVect(Y)$ of parabolic bundles on $Y$ is equivalent to the category of 
parabolic bundles with rational weights on $Y$ in the sense of \cite{MS}, \cite{MY}. When 
characteristic $p$ of $k$ is positive, the same proof verbatim gives us the following 
proposition.

\begin{proposition}\label{prop2}
The category of $\PVect^t(Y)$ is equivalent to the category of parabolic bundles in the 
sense of Mehta-Seshadri, \cite{MS}, with rational weights satisfying the condition that the denominators
in weights are coprime to the characteristic $p$.
\end{proposition}

Given a branch data $P$ on $Y$, and a covering $f:X\,\longrightarrow\, Y$, for a closed 
point $x\,\in\, X$, let $(f^*P)(x)$ be the compositum $P(f(x))K_{X,x}$. Then $f^*P$ is a 
branch data on $X$; it is defined by $x\, \longmapsto\, (f^*P)(x)$.

Given a covering $f\,:\,X\, \longrightarrow \,Y$, in \cite{KP} the pullback functor from 
the category $\Vect(Y,P)$ to the category $\Vect(X,f^*P$) was defined. This was used in 
\cite{KM} to define the functor $$\PVect(Y,P)\, \longrightarrow \,\PVect(X,f^*P).$$ Note 
that if $P$ is a tame branch data on $Y$ then $f^*P$ is also tame. Hence tame orbifold 
(respectively, parabolic) bundles on $Y$ pull back to tame orbifold (respectively, 
parabolic) bundles on $X$.

\section{Genuinely ramified covers relative to a branch data}

A separable surjective map $f\,:\,X\, \longrightarrow \,Y$ between irreducible smooth projective 
curves is called \textit{genuinely ramified} if the homomorphism of \'etale fundamental
groups
$$
f_*\, :\, \pi_1^{\rm et}(X)\,\longrightarrow\, \pi_1^{\rm et}(Y)
$$
induced by $f$ is surjective; see \cite[Definition 2.5]{BP} and \cite[Proposition 2.6]{BP}.

\begin{lemma}\label{genuinely-ramified-pullback}
Let $f\,:\,X\, \longrightarrow \,Y$ be a genuinely ramified map.
Let $g\,:\,Z\, \longrightarrow \,Y$ be a Galois cover such that $B_f(y)$ and $B_g(y)$ are
linearly disjoint over $K_{Y,y}$ for all $y\,\in \,Y$.
Suppose $k(Z)$ and $k(X)$ are linearly disjoint over $k(Y)$.
Let $W$ be the normalization of the fiber product $X\times_Y Z$.
Then the natural projection
$$
p_Z\, :\, W\, \longrightarrow \, Z
$$
is genuinely ramified. 
\end{lemma}

\begin{proof}
We first observe that the condition that $k(X)$ and $k(Z)$ are linearly disjoint over $k(Y)$ implies
that $W$ is connected.

Let $$\widetilde{f}\,:\,\widetilde{X}\, \longrightarrow \,Y$$ be the Galois closure of $f$ 
with Galois group $G$, and let $\widetilde{W}$ be the Galois closure of $p_Z$. We note 
that $\widetilde W$ is the normalization of $\widetilde{X}\times_Y Z$, and the Galois 
group of $k(\widetilde{W})/k(Z)$ is also $G$. Suppose $p_Z$ is not genuinely ramified. 
Then there exists an \'etale covering $Z'\, \longrightarrow \,Z$ dominated by $W$. Let 
$H\,=\,\Gal(k(\widetilde W)/k(Z'))$, and let $Y'$ be the normalization of $Y$ in $k(\widetilde 
X)^H$. Then $Y'$ is dominated by $X$. Moreover, for any point $y\,\in \,Y$ and a point 
$y'\,\in \,Y'$ lying above $y$, we know $K_{Y',y'}$ is contained in $B_f(y)$ and $B_g(y)$. 
Hence $Y'\, \longrightarrow \,Y$ is \'etale, which contradicts the
condition that $f$ is genuinely ramified (see \cite[Proposition 2.6]{BP}).
\end{proof}

\begin{lemma}\label{red.geom.branch.data}
Let $f\,:\,X\, \longrightarrow \,Y$ be a covering of smooth
projective curves. Let $P$ be a tame branch 
data on $Y$ such that $P(y)$ and $B_f(y)$ are linearly disjoint over $K_{Y,y}$ for
all $y\,\in \,Y$. 
Then there exists a Galois covering $g\,:\,Z\, \longrightarrow \,Y$ such that
\begin{enumerate}
\item $k(Z)$ and $k(X)$ 
are linearly disjoint over $k(Y)$,

\item $B_f(y)$ and $B_g(y)$ are linearly disjoint over $K_{Y,y}$ for all 
$y\,\in \,Y$, and

\item $P(y)\,=\, B_g(y)$ for all $y$ in support of $P$.
\end{enumerate}
\end{lemma}

\begin{proof}
Let $y_1$ be a point in $Y$ outside the support of $B_f$. Let $N$ be the least
common multiple of $[P(y):K_{Y,y}]$ for 
all $y\,\in \,Y$. Take a rational function $\alpha$ on $Y$ which has
\begin{itemize}
\item a simple zero at $y_1$,

\item a zero of multiplicity $N/P(y)$ at every $y$ in support of $P$ and

\item the 
remaining zeroes and poles are outside the support of $B_f$.
\end{itemize}
Let $g\,:\,Z\, 
\longrightarrow \,Y$ be the normalization of $Y$ in $k(Y)[\alpha^{1/N}]$. Note that $Z \, 
\longrightarrow \,Y$ is totally ramified at $y_1$, and $X\, \longrightarrow \,Y$ is \'etale 
over $y_1$, and hence $k(Z)$ and $k(X)$ are linearly disjoint over $k(Y)$. The rest of the 
statements follow as $g$ is ramified only at the zeros and poles of $\alpha$ and the 
ramification index of $g$ at a zero $y$ of $\alpha$ is $N/ord_y(\alpha)$.
\end{proof}

\begin{lemma}\label{lem3}
Let $P$ be a branch data on $Y$. Let $S$ be a finite subset of $Y\setminus \Supp(P)$. Then
there exists a geometric branch data $P'$ on $Y$ such that $P'(y)\,=\,P(y)$ for all $y$ in
the support of $P$ while the support of $P'$ is disjoint from $S$.
\end{lemma}

\begin{proof}
We first assume that $Y\,=\,\PP^1$. Consider $\Supp(P)\,=\,\{y_1,\,\cdots,\,y_r\}$. Let 
$n_i$ be the tame degree of $P(y_i)$, and $n_0\,:=\,\lcm \{n_i\,\,\mid\,\, 1\,\le\, 
i\,\le\, r\}$. Take any $y_0\,\in \,\PP^1$ outside $S\bigcup \Supp(P)$. Set $P'(y_0)$ to be the cyclic 
extension of $K_{Y,y_0}$ of degree $n_0$ and $P'(y)\,=\,P(y)$ for all $y\,\in \,Y$ 
different from $y_0$. Then $P'$ is geometric by \cite[Corollary 2.32]{KP}.
 
For the general case of $Y$, take a point $y_0\,\in \,Y$ outside $\Supp(P)$, and also
take a regular function $a$ on $Y\setminus y_0$ which takes different values at
the points of 
$\Supp(P)\bigcup S$. Then we get a covering $a\,:\,Y\, \longrightarrow \,\PP^1$ totally ramified 
at $y_0$. Let $Q(a(y))\,=\,P(y)$ for $y\,\in \,\Supp(P)$, and take $x_0\,\in \,\PP^1\setminus 
a(\Supp(P)\bigcup S\bigcup\{y_0\})$. Then by the $Y\,=\,\PP^1$ case there is a Galois covering
$b\,:\,Z\, \longrightarrow \,\PP^1$ such that $B_b(x)\,=\,Q(x)$ for $x\,\in \,\PP^1$ 
except at one point $x_0$. Now the normalized pullback $c\,:\,Z'\, \longrightarrow \,Y$ of 
$b$ is a connected Galois cover of $Y$ as $b$ is \'etale over $a(y_0)$ and $a$ is 
totally ramified at $y_0$. Set $P'\,=\,B_c$. Then $P'$ satisfies all the statements in
the lemma.
\end{proof}

\begin{lemma}\label{wild.red.geom.branch.data}
Let $f\,:\,X\, \longrightarrow \,Y$ be a genuinely ramified Galois covering of smooth
projective curves.
Let $P$ be a branch data on $Y$ 
such that $P(y)$ and $B_f(y)$ are linearly disjoint over $K_{Y,y}$ for all $y\,\in \,Y$.
Then there exists a Galois covering $g\,:\,Z\, \longrightarrow \,Y$ such that
\begin{enumerate}
\item $k(Z)$ and $k(X)$ are linearly 
disjoint over $k(Y)$,

\item $B_f(y)$ and $B_g(y)$ are linearly disjoint over $K_{Y,y}$ for all $y\,\in \,Y$, and 

\item $P(y)\,=\, B_g(y)$ for all $y$ in the support of $P$.
\end{enumerate}
\end{lemma}

\begin{proof}
Apply Lemma \ref{lem3} with $S\,=\,\Supp(B_f)\setminus \Supp(P)$ to obtain a geometric 
branch data $P'$ on $Y$ such that $P'(y)\,=\,P(y)$ for $y\,\in \,\Supp(P)$ and $S\subset 
Y\setminus \Supp(P')$. Let $g\,:\,(Z,\,0)\, \longrightarrow \,(Y,\,P')$ be a Galois 
\'etale covering of formal orbifolds. Then $g$ satisfy conclusion (2) and (3) of the lemma. Let $K\,=\,k(Z)\bigcap k(X)$, and let $a\,:\,W\, 
\longrightarrow \,Y$ be the normalization of $Y$ in $K$. Then $a\,:\,W\, \longrightarrow 
\,Y$ is \'etale. This is because $a$ is dominated by the coverings $f$ and $g$, the 
covering $g\,:\,Z\, \longrightarrow \,Y$ is \'etale over points in $S$ and $f$ is \'etale 
outside $B_f$, hence the branch locus of $a$ is contained in $\Supp(B_f)\bigcap \Supp(P)$. 
Now for a point $y$ in $\Supp(B_f)\bigcap \Supp(P)$ and $w\in W$ lying above $y$, the 
given conditions that $P(y)$ and $B_f(y)$ are linearly disjoint and $P'(y)\,=\,P(y)$
imply that the 
local field extension $K_{W,w}/K_{Y,y}$ is trivial. Hence points in $\Supp(B_f)\bigcap 
\Supp(P)$ are also not in the branch locus of $a$. Since $f$ is genuinely ramified $a$ is 
the identity map, i.e., $K\,=\,k(Y)$. Since $k(X)/k(Y)$ and $k(Z)/k(Y)$ are Galois, this implies 
$k(Z)$ and $k(X)$ are linearly disjoint over $k(Y)$.
\end{proof}

\section{Pullback of equivariant stable bundles} 

Let $X$ be a smooth projective curve equipped with an action of a finite group $\Gamma$.
A $\Gamma$--equivariant vector bundle $E$ on $X$ is called $\Gamma$--\textit{stable}
(respectively, $\Gamma$--\textit{semistable}) if for all $\Gamma$--equivariant subbundle
$0\, \not=\, F\, \subsetneq\, E$, the inequality
$$
\mu(F)\, :=\, \frac{{\rm degree}(F)}{{\rm rank}(F)}\, <\,
\mu(E)\, :=\, \frac{{\rm degree}(E)}{{\rm rank}(E)}\ \
\left(\text{respectively, }\ \mu(F)\,\leq \,\mu(E)\right)
$$
holds.

When $\Gamma\,=\, \{e\}$, a $\Gamma$--stable (respectively, $\Gamma$--semistable) vector
bundle is called stable (respectively, semistable). A semistable vector
bundle is called \textit{polystable} if it is a direct sum of stable vector bundles.

\begin{lemma}\label{lem1}
Let $E$ be a $\Gamma$--stable vector bundle on $X$. Then $E$ is polystable in the usual sense.
\end{lemma}

\begin{proof}
Note that $E$ a is semistable bundle on $X$ by \cite[Lemma 3.10]{KP}.
Let $W\, \subset\, E$ be the socle (see \cite[p.~23, Lemma 1.5.5]{HL}). From the 
uniqueness of the socle bundle we know that the action of $\Gamma$ on $E$ preserves $W$. 
Since $E$ is $\Gamma$--stable, this implies that $W\,=\, E$. Hence $E$ is polystable. 
\end{proof}

Let $X$ and $Y$ be irreducible smooth projective curves such that each of them
is equipped with an action of $\Gamma$, and let
\begin{equation}\label{f}
f\, :\, X\, \longrightarrow\, Y
\end{equation}
be a nonconstant separable $\Gamma$--equivariant morphism.
Given a $\Gamma$--equivariant bundle $E\, \longrightarrow\, Y$, the action of $\Gamma$
on $E$ produces an action of $\Gamma$ on $f^*E$, because the map $f$ is $\Gamma$--equivariant.
So $f^*E$ is a $\Gamma$--equivariant bundle on $X$.

\begin{proposition}\label{prop1}
Assume that the map $f$ in \eqref{f} is genuinely ramified. Let
$E$ be a $\Gamma$--stable vector bundle on $Y$. Then the pullback $f^*E$ is also
$\Gamma$--stable.
\end{proposition}

\begin{proof}
From Lemma \ref{lem1} we know that $E$ is polystable.
Let
$$
E\,=\, \bigoplus_{i=1}^\ell F_i
$$
be a decomposition of $E$ into a direct sum of stable vector bundles with
$\mu(F_i)\,=\, \mu(E)$ for all $1\, \leq\, i\, \leq\, \ell$. Therefore, we have
\begin{equation}\label{f2}
f^*E\,=\, \bigoplus_{i=1}^\ell f^*F_i\, .
\end{equation}
Now from \cite[Theorem 1]{BP} we conclude that each $f^*F_i$ is stable. Since
$$
\mu(f^*F_i)\,=\, \text{degree}(f)\cdot \mu(F_i)\, ,
$$
this implies that $f^*E$ is polystable.

Assume that $f^*E$ is not $\Gamma$--stable. Let
\begin{equation}\label{cs}
0\, \not=\, {\mathcal S}\, \subsetneq\, f^*E
\end{equation}
be a subbundle such that
\begin{itemize}
\item $\mu({\mathcal S})\,=\, \mu(f^*E)$, and

\item the action of $\Gamma$ on $f^*E$ preserves $\mathcal S$
\end{itemize}

We will prove that $\mathcal S$ is polystable. For that first note that if
$V_1$ and $V_2$ are stable vector bundles with $\mu(V_1)\,=\, \mu(V_2)$, then any
nonzero homomorphism $V_1\, \longrightarrow\, V_2$ is an isomorphism. From this it
follows that any polystable vector bundle $\mathcal V$ can be uniquely expressed as
\begin{equation}\label{j1}
{\mathcal V}\, =\, \bigoplus_{i=1}^n V_i\otimes_k \text{Hom}(V_i,\, {\mathcal V})\, ,
\end{equation}
where $\{V_i\}_{i=1}^n$ are all stable vector bundles such that
\begin{itemize}
\item $\mu({\mathcal V})\, =\, \mu(V_i)$, and

\item $\text{Hom}(V_i,\, {\mathcal V})\, \not=\, 0$.
\end{itemize}
If ${\mathcal T}\, \subset\, {\mathcal V}$ is a subbundle of the polystable
vector bundle $\mathcal V$ such that $\mu({\mathcal V})\, =\, \mu({\mathcal T})$, then
\begin{equation}\label{j2}
{\mathcal T}\,=\, \bigoplus_{i=1}^n {\mathcal T}\cap
\left(V_i\otimes_k \text{Hom}(V_i,\, {\mathcal V})\right)
\end{equation}
with respect to the unique decomposition in \eqref{j1}. Now any subbundle $W$
of $V_i\otimes_k \text{Hom}(V_i,\, {\mathcal V})$ with $\mu(W)\,=\, \mu(V_i)$
must be of the form $V_i\otimes_k B$, where $B\, \subset\, \text{Hom}(V_i,\, {\mathcal V})$
is a subspace. This implies that $W$ is polystable. Since each nonzero
${\mathcal T}\bigcap\left(V_i\otimes_k \text{Hom}(V_i,\, {\mathcal V})\right)$ in
\eqref{j2} is polystable with $\mu\left({\mathcal T}\bigcap\left(V_i\otimes_k
\text{Hom}(V_i,\, {\mathcal V})\right)\right)\,=\, \mu(V_i)\,=\, \mu({\mathcal V})$, we
conclude that ${\mathcal T}$ is polystable.

Therefore, $\mathcal S$ in \eqref{cs} is polystable.

Since each $f^*F_i$ in \eqref{f2} is stable, we know that the polystable
bundle ${\mathcal S}$ in \eqref{cs} admits an isomorphism
\begin{equation}\label{f3}
{\mathcal S}\,\stackrel{\sim}{\longrightarrow}\, \bigoplus_{j=1}^{\ell'} f^*F_{a_j}
\end{equation}
with $a_1\, <\, \cdots \, <\, a_{\ell'}$ and $\ell'\, <\, \ell$. Denote
$$
V\,=\, \bigoplus_{j=1}^{\ell'} F_{a_j}\, ;
$$
from \eqref{f3} it follows that 
\begin{equation}\label{f4}
{\mathcal S}\,\stackrel{\sim}{\longrightarrow}\, f^*V\, .
\end{equation}

Let $A$ and $B$ be two semistable vector bundles on $Y$ with $\mu(A)\,=\, \mu(B)$. From
\cite[Lemma 4.3]{BP} we know that the natural homomorphism
$$
H^0(Y,\, \text{Hom}(A,\, B))\, \longrightarrow\, H^0(X,\, \text{Hom}(f^*A,\, f^*B))
$$
is an isomorphism. Therefore, there is a homomorphism
$$
\psi\, :\, V\, \longrightarrow\, E
$$
such that $f^*\psi$ coincides with the inclusion map ${\mathcal S}\, \hookrightarrow\, f^*E$
in \eqref{cs} once the isomorphism in \eqref{f4} is invoked. Note that since
$V$ and $E$ are semistable with $\mu(V)\,=\, \mu(E)$, it follows that
$\psi(V)$ is a subbundle of $E$. Since $f^*\psi$ coincides with the inclusion map
${\mathcal S}\, \hookrightarrow\, f^*E$, it follows immediately that the
subbundle $\psi(V)\, \subset\, E$ has the property that its pullback
$$
f^*\psi(V)\, \subset\, f^*E
$$
coincides with the subbundle ${\mathcal S}\, \subset\, f^*E$ in \eqref{cs}.

Since ${\mathcal S}$ is preserved by the action of $\Gamma$ on $f^*E$, it follows 
immediately that the subbundle $\psi(V)\, \subset\, E$ is preserved by the action of 
$\Gamma$ on $E$. But this contradicts the given condition that $E$ is a $\Gamma$--stable 
vector bundle; note that $\mu(\psi(V))\,=\,\mu(E)$, because $\mu({\mathcal S})\,=\, 
\mu(f^*E)$. Therefore, we conclude that the pullback $f^*E$ is a $\Gamma$--stable bundle.
\end{proof}

\section{parabolic stability of the pullback bundle}

Let $E_*$ be a parabolic bundle on $Y$ with respect to $P$. Let
$$
h\, :\, (X,\, 0)\, \longrightarrow\, (Y,\, P)
$$
be an \'etale Galois covering of orbifolds with Galois group $\Gamma$. Let
$W$ be the $\Gamma$--equivariant vector bundle on $X$
corresponding to $E_*$.

The following proposition is standard.

\begin{proposition}\label{prop3}
The parabolic bundle $E_*$ is stable (respectively, semistable) if and only if
$W$ is $\Gamma$--stable (respectively, $\Gamma$--semistable). Similarly,
$E_*$ is polystable if and only if $W$ is $\Gamma$--polystable.
\end{proposition}

\begin{proof}
This is a tautology when the base field has positive characteristic. For the characteristic zero case
we have $\text{degree}(h)\cdot\text{par-deg}(E_*)\,=\, \text{degree}(W)$.
Also, the subbundles of $E_*$ with induced parabolic structure correspond to the
$\Gamma$--equivariant subbundles of $W$ (see \cite{Bi}). The proposition follows
immediately from these.
\end{proof}

Let $f\,:\,X\, \longrightarrow \,Y$ be a genuinely ramified map between
smooth projective curves. Let $P$ be a branch data on $Y$. Let $V$ be a parabolic bundle on $Y$ 
with respect to $P$. Then $f^*V$ is a parabolic bundle on $X$ with respect to $f^*P$.

\begin{theorem}\label{thm1}
Suppose $B_f(y)$ and $P(y)$ are linearly disjoint over $K_{Y,y}$ for all $y\,\in \,Y$.
Also assume that either $P$ is tame or $f$ is Galois. Then the pullback, by $f$, of every
stable parabolic bundle on $Y$ with respect to $P$ is a stable parabolic
bundle on $X$ with respect to $f^*P$. 
\end{theorem}

\begin{proof}
By Lemma \ref{red.geom.branch.data} and Lemma \ref{wild.red.geom.branch.data} there exists a 
$\Gamma$--Galois covering $g\,:\,Z\, \longrightarrow \,Y$ such that
\begin{itemize}
\item $B_g(y)\,=\,P(y)$ for all 
$y\,\in \,\Supp(P)$, and

\item $B_f(y)$ and $B_g(y)$ are linearly disjoint.
\end{itemize}
Moreover, the normalization $W$ of $Z\times_Y X$ is connected. By Lemma 
\ref{genuinely-ramified-pullback}, the morphism $f_Z:W\, \longrightarrow \,Z$ is genuinely 
ramified. Also note that $f_Z$ is $\Gamma$--equivariant. Now replacing $P$ by $B_g$ we
may assume that $P$ is geometric branch 
data; this is because $P\,\le \,B_g$. Hence a parabolic bundle with respect to $P$ is 
also a parabolic bundle with respect to $B_g$. Consequently, $V$ corresponds to a 
$\Gamma$--equivariant bundle $\widetilde V$ on $Z$. Now by Proposition \ref{prop1},
the pullback $f_Z ^*\widetilde 
V$ is a stable $\Gamma$--equivariant bundle on $W$ (see Proposition
\ref{prop3}). But $(W,\,0)\, \longrightarrow 
\,(X,\,f^*B_g)$ is a $\Gamma$--Galois \'etale cover of formal
orbifolds (\cite[Proposition 2.16]{KP}). So the 
$\Gamma$--bundle $f_Z ^*\widetilde V$ on $W$ corresponds to the parabolic bundle $f^*V$ on $X$ 
with respect to $f^*P$. Hence $f^*V$ is a stable parabolic bundle on $Y$.
\end{proof}

\begin{remark}
The additional conditions in Theorem \ref{thm1} imply that the morphism $f$ induce a 
surjection $\pi_1(X,\,f^*P)\,\longrightarrow\, \pi_1(Y,\,P)$. Though we expect the result
to be true under 
a milder hypothesis, the proof in positive characteristic is not yet clear to us. The 
case of characteristic zero has been dealt with in \cite{BKP}.
\end{remark}

\section*{Acknowledgements}

We thank the referee for helpful comments. The first and third authors thank
Indian Statistical Institute for hospitality while the work was carried out.

\end{document}